\documentclass[11pt,reqno]{amsart}
\usepackage{iftex}
\usepackage[T1]{fontenc}
\ifPDFTeX\usepackage[utf8]{inputenc}\fi
\usepackage{lmodern,microtype}
\usepackage{amsmath,amssymb,mathtools}
\usepackage{booktabs,array,longtable}
\usepackage{xcolor}
\usepackage{hyperref}
\hypersetup{colorlinks=true,linkcolor=blue!50!black,
  citecolor=green!35!black,urlcolor=blue!60!black,
  pdftitle={Nonexistence of a Strongly Regular Graph with Parameters (266,45,0,9)},
  pdfauthor={Kay Akiyama}}
\newtheorem{theorem}{Theorem}[section]
\newtheorem{proposition}[theorem]{Proposition}
\newtheorem{lemma}[theorem]{Lemma}

\theoremstyle{definition}

\newtheorem{remark}[theorem]{Remark}
\newcommand{\Z}{\mathbb Z}
\newcommand{\Q}{\mathbb Q}

\newcommand{\F}{\mathbb F}
\newcommand{\one}{\mathbf 1}
\newcommand{\ip}[2]{\langle #1,#2\rangle}
\newcommand{\norm}[1]{\lVert #1\rVert}
\newcommand{\rank}{\operatorname{rank}}
\newcommand{\tr}{\operatorname{tr}}
\newcommand{\SRG}{\operatorname{SRG}}
\newcommand{\pr}{\operatorname{pr}}
\newcommand{\disc}{\operatorname{disc}}
\newcommand{\code}[1]{\texttt{\detokenize{#1}}}
\newcommand{\leanref}[2]{\par\smallskip\noindent
  {\footnotesize\raggedright\emph{Formal reference.} \path{#1}\par\nobreak
  \noindent\code{#2}.\par}\smallskip}
\title[Nonexistence of an $\SRG(266,45,0,9)$]{Nonexistence of a Strongly Regular Graph with Parameters
  $(266,45,0,9)$: A Certificate-Free Lean Proof}
\author{Kay Akiyama}
\date{8 September 2026}
\subjclass[2020]{Primary 05E30; Secondary 05C50, 05B05, 11H55, 68V20}
\keywords{strongly regular graph, integral lattice, even unimodular lattice,
  harmonic theta series, quasi-symmetric design, formal proof, Lean}

\begin{document}
\raggedbottom
\begin{abstract}
We prove that no strongly regular graph with parameters $(266,45,0,9)$ exists.
The proof is formalized in Lean~4 and Mathlib without external infeasibility
certificates or assumed classification theorems. A hypothetical graph gives
a rank-$12$ integral Gram lattice with an integral centroid. A Lorentzian
change of form, a marked $D_7$ gluing, and an explicit rank-six complement
produce a positive-definite even unimodular lattice of rank~$24$, together
with the original indexed family of $220$ vectors. Harmonic theta identities
and a root-isolation inequality force the root system
$A_{11}\perp D_7\perp E_6$. First and second moments then exclude the possible
complements: the final case reduces to an impossible binary projection
identity $4x+4y-2z=50$. A type-$A$ subcase is closed by a separate
classification-free proof of the known nonexistence of a quasi-symmetric
$2$-$(56,12,9)$ design with intersections $0,3$. That argument constructs a
Krein graph and forces a Steiner $3$-$(12,4,1)$ design, contradicting its
replication equation. The formal theorem depends only on the three standard
Lean axioms and has also been checked independently with nanoda. The archived
formalization is release v2.0.0.
\end{abstract}
\maketitle

\section{Introduction}\label{sec:introduction}

A finite simple graph is strongly regular with parameters $(v,k,\lambda,\mu)$
if it has $v$ vertices, each vertex has $k$ neighbours, adjacent vertices have
$\lambda$ common neighbours, and distinct nonadjacent vertices have $\mu$
common neighbours. Standard spectral and divisibility conditions leave many
feasible parameter sets whose existence is a separate question; see
\cite{BrouwerVanMaldeghem2022} and Brouwer's online parameter tables
\cite{BrouwerTables}.

\begin{theorem}\label{thm:main}
No strongly regular graph has parameters $(266,45,0,9)$.
\end{theorem}

For these parameters the adjacency spectrum would be
\[
45^{(1)},\qquad 3^{(209)},\qquad (-12)^{(56)},
\]
and the Hoffman ratio bound for a coclique is $56$
\cite[Section~3.5]{BrouwerHaemers2012}. Earlier results exclude a
$5$-chromatic graph with these parameters
\cite[Theorem~7.3, p.~3089]{FialaHaemers2006}
and, through the associated quasi-symmetric design, a graph containing a
Delsarte coclique of size~$56$
\cite{AdmEtAl2018,MunemasaTonchev2020}. Theorem~\ref{thm:main} makes neither
assumption. Its contribution is the exclusion of an arbitrary hypothetical
graph by a lattice construction retaining its local incidence data.

\begin{remark}
An earlier unpublished draft, synchronized with release v1.1.0, used
rank-$15$ odd unimodular hosts and finite infeasibility certificates.
The present argument retains the local Gram construction and integral
centroid, but replaces those exclusions by a marked rank-$24$ completion
and moment identities.
\end{remark}

\subsection*{The mathematical mechanism}
Fix a vertex of the hypothetical graph. Its second subconstituent has
$220$ vertices, which index norm-$3$ vectors in a rank-$12$ lattice.
Subtracting the centroid direction gives a zero-sum tight frame in
dimension~$11$. The arithmetic of this frame permits a marked completion
in dimension~$24$. The main route is
\[
\begin{gathered}
\SRG(266,45,0,9)\ \Longrightarrow\
\text{integral centroid and rank-$11$ tight frame}\\
\Longrightarrow\ \text{marked even unimodular rank-$24$ lattice}\\
\Longrightarrow\ A_{11}\perp D_7\perp E_6
\ \Longrightarrow\ \text{component masses }275,220\\
\Longrightarrow\ \text{binary projection obstruction}.
\end{gathered}
\]
One branch forces a binary weight-three factorization of the original
Gram matrix. It is excluded independently by
\[
\begin{gathered}
\text{quasi-symmetric }2\text{-}(56,12,9),\ \{0,3\}\\
\Longrightarrow\ \text{canonical biplane and a }3\text{-}(57,12,2)\text{ design}\\
\Longrightarrow\ \text{structured Krein graph on }324\text{ vertices}\\
\Longrightarrow\ \text{Steiner }3\text{-}(12,4,1)
\ \Longrightarrow\ 3r=55.
\end{gathered}
\]
The extra dimensions are useful because they impose identities on the
\emph{same} configuration that came from the graph. An unmarked existence
statement about a lattice would not be sufficient. In particular, the
$220$ indices are never replaced by a set of distinct projected vectors.

\subsection*{Prior results and scope}
The quasi-symmetric-design nonexistence statement is due to Munemasa and
Tonchev \cite{MunemasaTonchev2020}; it is not claimed as new.
Their proof uses the Hall--Connor embedding theorem and the classification
of biplanes on $56$ points \cite{HallConnor1954,KaskiOstergard2008},
and builds on the ternary-code computations of Key and Tonchev
\cite{KeyTonchev1997}. Our internal proof constructs a
compatible biplane from the hypothetical design and closes the argument by
local graph identities, without using that classification.

Gavrilyuk and Makhnev \cite{GavrilyukMakhnev2005} proved that no graph
with parameters $(324,57,0,12)$ exists. Section~\ref{sec:design}
establishes the structured special case needed here, not a new general
$324$-vertex theorem. Conversely, the classical subconstituent theorem
\cite{CameronGoethalsSeidel1978,Biggs2009} implies that every hypothetical
$\SRG(324,57,0,12)$ has second subconstituents with parameters
$(266,45,0,9)$.
Together with Theorem~\ref{thm:main}, this recovers the known general
$324$-vertex nonexistence. This last implication is literature-level
context, not an additional formalized corollary or an input to our proof.

Chen, Makhnev, and Nirova \cite{ChenMakhnevNirova2023} characterize the
existence of a triangle-free Krein graph $\operatorname{Kre}(r)$ by the
existence of a graph $\operatorname{Kre}(r)'$ with the second-subconstituent
parameters that is also the complement of a quasi-symmetric design's
block graph. For $r=3$ these are the $324$- and $266$-vertex parameter
sets above. The block-graph realization is an additional hypothesis;
their result does not extend an arbitrary $\SRG(266,45,0,9)$ to a
$324$-vertex graph. Section~\ref{sec:design} supplies such a realization
only in the design branch.

Euclidean representations and Gram matrices are standard tools in the
study of strongly regular graphs
\cite[Section~1.3.5]{BrouwerVanMaldeghem2022}. Direct precedents for
nonexistence arguments using these tools are Bondarenko, Prymak, and
Radchenko's exclusion of $(76,30,8,14)$
\cite{BondarenkoPrymakRadchenko2017} and Bondarenko et al.'s exclusion of
$(460,153,32,60)$ \cite{BondarenkoEtAl2018}. Both combine Euclidean
representations and Gram-matrix restrictions with bounds on $4$-cliques;
the former also uses low-dimensional orthogonal projections.
Here the additional ingredient is an integral marked completion retaining
the original indexed configuration, to which root and moment identities apply.
We use standard terminology for lattices and root systems
\cite{ConwaySloane1999}. We write squared norms as $\norm{v}^2$;
as usual in lattice terminology, a vector of norm $m$ means
$\norm{v}^2=m$. No classification of rank-$24$ even unimodular
lattices is invoked. The required ADE analysis and harmonic theta argument
are proved in the formal development. The paper gives their mathematical
content and the reductions they support; Appendix~\ref{app:formal-map}
identifies the corresponding formal statements.

\section{The local Gram lattice}\label{sec:local}

Assume that $\Gamma$ is a hypothetical graph in Theorem~\ref{thm:main}.
Its adjacency matrix $A$ satisfies
\begin{equation}\label{eq:srg}
A^2=36I-9A+9J.
\end{equation}
The spectrum stated in the introduction follows by restricting this
identity to $\one^\perp$ and using the trace. Fix a vertex $o$ and put
$X=\Gamma(o)$ and $Y=V(\Gamma)\setminus(\{o\}\cup X)$.
Then $|X|=45$, $|Y|=220$, and $X$ is independent.
For $y\in Y$ let $B_y=\Gamma(y)\cap X$.

\begin{lemma}\label{lem:local-design}
The indexed blocks $(B_y)_{y\in Y}$ form a $2$-$(45,9,8)$ design with
replication number $44$. The induced graph $H=\Gamma[Y]$ is $36$-regular.
Writing $M$ for the $45\times220$ incidence matrix and $S=M^{\mathsf T}M$,
one has
\begin{align}
S\one&=396\one,&H\one&=36\one,\nonumber\\
S^2&=36S+648J,&SH=HS&=81J-9S,\label{eq:local-algebra}\\
H^2&=36I+9J-S-9H.&&\nonumber
\end{align}
\end{lemma}
\begin{proof}
A vertex of $Y$ has $9$ common neighbours with $o$, so $|B_y|=9$.
Each point of $X$ has $44$ neighbours in $Y$. Two distinct points of $X$
have $8$ common neighbours in $Y$, the ninth being $o$.
Thus $MM^{\mathsf T}=36I+8J$, giving the identity for $S^2$.
The other identities are the corresponding blocks of~\eqref{eq:srg}.
\end{proof}

Define
\begin{equation}\label{eq:gram}
L=9I+3J-S-3H.
\end{equation}
Its diagonal is $3$, and for $y\ne z$ its entry is $0$ if $y\sim_H z$,
and $3-|B_y\cap B_z|$ otherwise. Substitution in
\eqref{eq:local-algebra} gives
\begin{equation}\label{eq:gram-spectrum}
L\one=165\one,\qquad L^2=45L+90J.
\end{equation}
Consequently $L$ is positive semidefinite, with nonzero eigenvalues
$165^{(1)}$ and $45^{(11)}$. Indeed, the possible eigenvalues on
$\one^\perp$ are $0,45$, and $\tr L=660$ fixes the multiplicity.
Let $\Lambda$ be the quotient of $\Z^Y$ by its integral radical for $L$,
and let $g_y$ be the image of the $y$th standard basis vector. This is a
positive-definite integral lattice of rank~$12$, with
$\ip{g_y}{g_z}=L_{yz}$.

\begin{lemma}\label{lem:intersection}
For $y\ne z$ one has $|B_y\cap B_z|\le3$. In particular
$L_{yz}\in\{0,1,2,3\}$.
\end{lemma}
\begin{proof}
The symmetric matrix $Q=4L-3J$ satisfies $Q^2=180Q$, so its eigenvalues
are $0$ or $180$ and it is positive semidefinite. Its diagonal entries
are $Q_{yy}=9$. If $y,z$ are nonadjacent and
$t=|B_y\cap B_z|$, the corresponding principal two-by-two minor gives
$|9-4t|\le9$, hence $t\le4$.

To exclude $t=4$, put $K_0=L-J=(k_{yz})$. Then
\[
K_0^2=45K_0+25J,\qquad K_0\one=-55\one,\qquad k_{yy}=2.
\]
All off-diagonal entries of $K_0$ are integers between $-2$ and $2$.
For $\phi(a)=a(a+1)/2$ and $w\notin\{y,z\}$,
\[
\phi(k_{yw})+\phi(k_{zw})+k_{yw}k_{zw}-H_{yw}H_{zw}\ge0.
\]
Without the last term the expression is $\phi(k_{yw}+k_{zw})\ge0$.
If the last term is $1$, both $k$-entries are $-1$ and the expression
is zero. Under $k_{yz}=-2$, set $Y'=Y\setminus\{y,z\}$. The row
and diagonal square identities give
\[
\sum_{w\in Y'}k_{yw}=-55,\qquad
\sum_{w\in Y'}k_{yw}^2=115-2^2-(-2)^2=107.
\]
Hence $\sum_{w\in Y'}\phi(k_{yw})=(107-55)/2=26$, and likewise
for $z$. The off-diagonal identity $(K_0^2)_{yz}=-65$ gives
$\sum_{w\in Y'}k_{yw}k_{zw}=-65+8=-57$.
There are $9-4=5$ common neighbours in $H$. Summation therefore gives
$0\le26+26-57-5=-10$.
\end{proof}

\begin{proposition}[Integral centroid]\label{prop:centroid}
There is $c\in\Lambda$ such that
\begin{equation}\label{eq:centroid}
\sum_{y\in Y}g_y=11c,\qquad
\ip{c}{g_y}=15,\qquad \norm{c}^2=300.
\end{equation}
\end{proposition}
\begin{proof}
Reduce $L$ modulo $11$ and write $\bar L$ for the resulting matrix.
Its rank is at most its rational rank, namely~$12$.
Equation~\eqref{eq:gram-spectrum} gives
$\bar L^2=\bar L+2J$, $\bar LJ=J\bar L=0$, and $J^2=0$.
Thus $\bar L^2$ is idempotent. It is nonzero (its diagonal is $5$), and its
trace is zero in $\F_{11}$. Its rank is a positive multiple of $11$,
so it is~$11$. If $\bar L$ also had rank $11$, its image would equal that
of $\bar L^2$. But the nonzero image of $\bar L^2-\bar L=2J$ lies both
in that image and in the kernel of the idempotent $\bar L^2$, a
contradiction. Hence $\rank\bar L=12$.

The integral radical is saturated. Its reduction injects into $\F_{11}^Y$
and has dimension $208$, so it is the entire kernel of $\bar L$.
Since $\bar L\one=0$, there is an integral radical vector $z$ with
$z\equiv\one\pmod{11}$. Writing $\one-z=11a$ gives
$c=\sum_y a_yg_y$. The remaining identities follow from the row sum of $L$.
\end{proof}
\leanref{SRG266/KernelReduction.lean}{exists_integral_centroid}

\section{A marked completion in dimension twenty-four}\label{sec:completion}

All rational spaces in this section are subsequently extended to real
Euclidean spaces when orthogonal projection is used. The constructions
and their lattice membership statements are first made over $\Q$.
The discriminant-form and gluing framework is standard
\cite[Section~1]{Nikulin1980}. In particular, even overlattices correspond
to isotropic subgroups of the discriminant form; their discriminant forms
are the induced forms on the orthogonal quotients. We establish the
instances needed below inside Lean.

\subsection{The Lorentzian form and the tight frame}
On $\Lambda\otimes\Q$ define
\[
h(v)=\frac{\ip{c}{v}}{15},\qquad
[v,w]=\ip{v}{w}-h(v)h(w),\qquad \rho=\frac c5.
\]
This form has signature $(11,1)$ and
\begin{equation}\label{eq:lorentz}
[g_y,g_z]=L_{yz}-1,\quad [\rho,g_y]=-1,\quad
[\rho,\rho]=-4,\quad [\rho,v]=-h(v).
\end{equation}
The lattice $\Lambda+\Z\rho$ is even and integral for $[\ ,\ ]$.
Its dual quotient is killed by $45$, as follows from the frame identity
\begin{equation}\label{eq:lorentz-frame}
\sum_y[v,g_y]g_y=45v-5h(v)c.
\end{equation}
For $v=g_z$, this identity follows by pairing both sides with each
$g_w$ in the original positive-definite form and using
$L^2=45L+90J$ and $\ip{c}{g_w}=15$; it then extends by linearity
because the $g_z$ span the space.
Choose a maximal integral overlattice $P$ of this marked lattice.
Since the dual quotient is killed by the odd integer $45$, every such
overlattice has odd index. For $v\in P$, an odd multiple $nv$ lies in
the original even lattice; thus $n^2[v,v]\in2\Z$, and integrality forces
$[v,v]\in2\Z$. Consequently $P$ is even, and its discriminant group
has no $2$-primary part. By the overlattice--isotropic-subgroup
correspondence, maximality makes its discriminant form anisotropic, and its
discriminant group is killed by~$15$.

Put $U=\rho^\perp=\ker h$ and $T=P\cap U$.
The restriction of $[\ ,\ ]$ to $U$ is positive definite and agrees with
the original form. Define
\begin{equation}\label{eq:frame}
u_y=g_y-\frac\rho4=g_y-\frac c{20}.
\end{equation}
These vectors belong to $T^*$, lie in a single coset of $T$, and satisfy
\begin{equation}\label{eq:tight-frame}
\norm{u_y}^2=\frac94,\qquad \sum_yu_y=0,\qquad
\sum_y\ip{v}{u_y}\ip{w}{u_y}=45\ip{v}{w}\quad(v,w\in U).
\end{equation}
In particular $\dim U=11$ and the vectors span~$U$.
Their common class has order $4$ in $T^*/T$, since
$4u_y=4g_y-\rho\in T$, while an integral multiple of $u_y$ can have
even norm only if that multiple is divisible by~$4$.
To see that this is the whole $2$-primary part, note that
$[\rho,g_y]=-1$ makes $\rho$ primitive in $P$ and gives
$[\rho,P]=\Z$. Thus the projection of $P$ onto $\Q\rho$ is
$\tfrac14\Z\rho$, so $[P:T\perp\Z\rho]=4$. Taking absolute
Gram determinants yields
\[
4|T^*/T|=4^2|P^*/P|,\qquad |T^*/T|=4|P^*/P|.
\]
Since $|P^*/P|$ is odd, the $2$-primary part of $T^*/T$ is precisely
the cyclic group generated by the common class of the $u_y$, of
quadratic value $9/4$ modulo $2\Z$.
\leanref{SRG266/Completion/Perpendicular.lean}{lorentzPerpGenerator_secondMoment}

\subsection{The marked \texorpdfstring{$D_7$}{D7}}
Use the coordinate lattice
\[
D_7=\{a\in\Z^7:\textstyle\sum_i a_i\text{ is even}\},\qquad
s=(1/2,\ldots,1/2).
\]
Its spinor class has order $4$ and norm $7/4$.
Fix one index $y_0$ and form the even gluing
\begin{equation}\label{eq:d7glue}
N=(T\perp D_7)+\Z(u_{y_0},s)\subset U\perp\Q^7.
\end{equation}
Because $u_y-u_{y_0}\in T$, all the vectors
$t_y=(u_y,s)$ belong to~$N$. They have norm~$4$ and
\begin{equation}\label{eq:lifted-pair}
\ip{t_y}{t_z}=L_{yz}+1.
\end{equation}
The order-four discriminant parts cancel. The remaining discriminant
form has only $3$- and $5$-primary parts, inherited from the anisotropic
discriminant form of~$P$.

\begin{proposition}[Marked completion]\label{prop:completion}
There is a positive-definite even unimodular lattice $\mathcal L$ in
\[
(U\perp\Q^7)\perp W,\qquad \dim W=6,
\]
containing $N\perp C$, where $C$ is one of the four lattices in
Table~\ref{tab:complements}. Moreover
\[
\mathcal L\cap(U\perp\Q^7)=N,\qquad
\mathcal L\cap W=C,\qquad \mathcal L\cap\Q^7=D_7.
\]
The $220$ specified vectors $(u_y,s,0)$ belong to~$\mathcal L$.
\end{proposition}

\begin{table}[ht]
\centering
\caption{The four explicit rank-six complements. Quadratic values in the
middle columns describe $\disc N$ and are taken modulo $2\Z$.
Angle brackets denote cyclic forms, and $0$ denotes the trivial group.}
\label{tab:complements}
\begin{tabular}{@{}llll@{}}
\toprule
$C$ & $3$-primary part & $5$-primary part & $\det C$\\
\midrule
$E_6$ & $\langle2/3\rangle$ & $0$ & $3$\\
$H_6$ & $\langle2/3\rangle$ & $\langle2/5\rangle$ & $15$\\
$A_2\perp A_4$ & $\langle4/3\rangle$ & $\langle4/5\rangle$ & $15$\\
$A_4\perp Q_{15}$ & $\langle4/3\rangle$ &
 $\langle2/5\rangle\perp\langle4/5\rangle$ & $75$\\
\bottomrule
\end{tabular}
\end{table}

Here $Q_{15}$ has Gram matrix $\left(\begin{smallmatrix}4&1\\1&4\end{smallmatrix}\right)$.
For an explicit description of $H_6$, take $A_5\perp\Z e$ with
$\norm{e}^2=10$ and adjoin $(3\omega,e/2)$, where $\omega$ generates
$A_5^*/A_5$ and $\norm{\omega}^2=5/6$. This is an index-two even extension,
of determinant $6\cdot10/4=15$; it contains the displayed $A_5$.

\begin{proof}[Proof of Proposition~\ref{prop:completion}]
An anisotropic quadratic form over $\F_3$ or $\F_5$ has dimension at
most two. Diagonalization leaves, at each prime, the trivial form, the
two one-dimensional square classes, or the anisotropic binary form.
The formal proof establishes these small normal forms directly, including
the lifts of their generators and their norms modulo $2\Z$.

One may first cancel each primary form by an explicit positive even
lattice, without prescribing the final dimension. The forms and their
cancelling models are listed in Table~\ref{tab:primary-models}.
Let $\omega_n\in A_n^*$ be an endpoint fundamental weight, of norm
$n/(n+1)$.
Here $\Xi_8$ is the index-four extension of $A_7\perp\Z e$,
$\norm{e}^2=10$, obtained by adjoining $(2\omega_7,e/2)$.
It has determinant $5$.
The lattice $\Xi_4$ is the index-three extension of $A_2\perp A_2(5)$
obtained by adjoining $(\omega_2,\omega_2)$, where $A_2(5)$ denotes
the form scaled by $5$. It has determinant $25$. Direct dual
calculations verify the asserted cancellations.

\begin{table}[ht]
\centering
\caption{Positive even lattices whose discriminant forms are the
negatives of the listed primary forms.}
\label{tab:primary-models}
\begin{tabular}{@{}cllr@{}}
\toprule
$p$ & Primary form & Cancelling model & Rank\\
\midrule
$3$ & $0$ & $0$ & $0$\\
$3$ & $\langle2/3\rangle$ & $E_6$ & $6$\\
$3$ & $\langle4/3\rangle$ & $A_2$ & $2$\\
$3$ & $\langle2/3\rangle\perp\langle2/3\rangle$ & $A_2\perp A_2$ & $4$\\
$5$ & $0$ & $0$ & $0$\\
$5$ & $\langle2/5\rangle$ & $\Xi_8$ & $8$\\
$5$ & $\langle4/5\rangle$ & $A_4$ & $4$\\
$5$ & $\langle2/5\rangle\perp\langle4/5\rangle$ & $\Xi_4$ & $4$\\
\bottomrule
\end{tabular}
\end{table}

The resulting positive-definite even unimodular lattice must have rank
divisible by $8$; see \cite[Chapter~2]{Ebeling2013}.
Applied to the source rank~$18$, this leaves exactly the four pairs in
Table~\ref{tab:complements}. The rank divisibility is itself proved in
the development, not assumed.

For each remaining pair, Table~\ref{tab:complements} supplies a rank-six
lattice with the negative discriminant form. In particular, $H_6$
replaces the preliminary $E_6\perp\Xi_8$, of rank $6+8=14$, by a
rank-six lattice with the same discriminant form. For the last pair,
$A_4\perp Q_{15}$ replaces $A_2\perp\Xi_4$, both of rank~$6$.
Indeed, in $Q_{15}$ the dual vectors
$(1/3,-1/3)$ and $(1/5,1/5)$ have norms $2/3$ and $2/5$ and are
orthogonal. Together with the dual of $A_4$, they give the last row's
opposite primary forms after multiplying the $5$-primary generators by~$2$.
The other rows are verified in the same way from their coordinate models.
Glue the full discriminant groups along these opposite forms. The graph
of the gluing is maximal isotropic, so the resulting lattice is even
and self-dual. Its intersection with either rational summand is exactly
the prescribed summand lattice. This also preserves the earlier marked
$D_7$ intersection and all the vectors $t_y$.
\end{proof}
\leanref{SRG266/Completion/Rank24Completion.lean}{exists_rank24_completion}

\section{The root system is forced}\label{sec:roots}

A root means a lattice vector of squared norm~$2$.
The first step is to show that the marked $D_7$ cannot acquire roots
joining it to the other summands.

\begin{lemma}[Root isolation]\label{lem:isolation}
Every root of $\mathcal L$ is either a root of the marked $D_7$ or is
orthogonal to its entire seven-dimensional span.
\end{lemma}
\begin{proof}
Write a root as $(v,a,w)\in U\perp\Q^7\perp W$.
Integrality against $D_7$ implies $a\in D_7^*$.
If $0<\norm{a}^2<2$, the elementary coordinate description of $D_7^*$ gives
either $a=\pm e_i$, of norm~$1$, or seven half-integral coordinates
$\pm1/2$, of norm~$7/4$.
Since $(v,a,w)$ pairs integrally with every $(u_y,s,0)$,
\[
\ip{v}{u_y}+\ip{a}{s}\in\Z.
\]
In the first case $\ip{a}{s}\in\Z+1/2$; in the second it is an odd
multiple of $1/4$. Thus, with $d=2-\norm{a}^2>0$, every index satisfies
\[
\ip{v}{u_y}^2\ge d/4.
\]
On the other hand, the root norm and~\eqref{eq:tight-frame} give
\[
\sum_y\ip{v}{u_y}^2=45\norm{v}^2\le45d.
\]
This would imply $220\le180$. Therefore a nonzero $a$ has norm~$2$,
and positive definiteness forces $v=w=0$. The marked intersection in
Proposition~\ref{prop:completion} then gives $a\in D_7$.
\end{proof}

The next identity is due to Venkov \cite{Venkov1999}. Together with the
ADE second moments, it implies that a nonempty rank-$24$ root system has
full rank and a common Coxeter number. We reproduce the argument needed
for our marked lattice, rather than invoke the Niemeier classification.

\begin{lemma}[Venkov's harmonic root identity]\label{lem:harmonic}
For the root set $R(\mathcal L)$,
\begin{equation}\label{eq:root-moment}
\sum_{r\in R(\mathcal L)}\ip{r}{v}\ip{r}{w}
=\frac{|R(\mathcal L)|}{12}\ip{v}{w}.
\end{equation}
\end{lemma}
\begin{proof}
For fixed real vectors $v,w$, the polynomial
\[
p_{v,w}(x)=24\ip{x}{v}\ip{x}{w}-\ip{v}{w}\norm{x}^2
\]
is homogeneous and harmonic of degree~$2$. By the harmonic-theta
transformation theorem of Hecke and Schoeneberg, its theta series on an
even unimodular lattice of rank~$24$ is a level-one cusp form of
weight~$14$; see \cite[Chapter~3]{Ebeling2013}.
The space of such cusp forms is zero: division by the discriminant cusp
form identifies it with the space of weight-two modular forms, which is
zero \cite[Chapter~2]{Ebeling2013}.
The coefficient of $q$ in this vanishing theta series is precisely
$24\sum_r\ip{r}{v}\ip{r}{w}-2|R(\mathcal L)|\ip{v}{w}$.

In Lean, the analytic step includes the Gaussian Fourier transform,
Poisson summation, convergence of the weighted theta series, and its
modular transformation and cusp conditions. Only then is Mathlib's
level-one theory applied, specifically
\path{CuspForm.discriminantEquiv} and
\path{ModularForm.levelOne_weight_two_rank_zero}.
These belong to Chris Birkbeck's development of level-one dimension
formulas for the sphere-packing formalization
\cite[Section~2.3]{HariharanEtAl2026}.
Our application is in
\path{SRG266/Lattice/LevelOneWeight14.lean}.
The root identity is therefore not a placeholder for a classification theorem.
\end{proof}

\begin{proposition}\label{prop:root-system}
The lattice $\mathcal L$ has $288$ roots, which span its rational space.
Its root system is $A_{11}\perp D_7\perp E_6$, with the $D_7$ factor
equal to the marked factor.
\end{proposition}
\begin{proof}
For $D_7$, summing over the roots $\pm e_i\pm e_j$ gives
\[
\sum_{r\in R(D_7)}\ip{r}{a}\ip{r}{b}=24\ip{a}{b}.
\]
Restrict~\eqref{eq:root-moment} to the marked span and use
Lemma~\ref{lem:isolation}. It follows that $|R(\mathcal L)|/12=24$,
so there are $288$ roots. The now positive-definite second moment also
shows that the roots span.

An integral norm-two root system is simply laced. Its simple-root graph
is a disjoint union of positive-definite ADE diagrams; for the standard
root-system classification and Coxeter numbers, see
\cite[Chapters~4--6]{Bourbaki2002}. The second-moment
scalar on an irreducible component is twice its Coxeter number, so every
component here has Coxeter number~$12$. The only possibilities are
$A_{11}$, $D_7$, and $E_6$. Since the isolated marked $D_7$ is a full
component, the remaining rank~$17$ must be $11+6$. This gives the stated
layout. The root counts $132+84+72=288$ agree with the moment calculation.
\end{proof}
\leanref{SRG266/Completion/RootLayout.lean}{exists_rank24_root_layout}

\section{First moments, component masses, and the final obstruction}
\label{sec:moments}

Fix the completion and a root-coordinate isometry supplied by
Proposition~\ref{prop:root-system}. Denote the real spans of $A_{11}$ and
$E_6$ by $E_A$ and $E_E$. The marked identification gives two orthogonal
decompositions of the same space:
\begin{equation}\label{eq:two-splittings}
D_7^\perp=U\perp W=E_A\perp E_E.
\end{equation}
For each original index $y$ put
$a_y=\pr_{E_A}t_y$ and $e_y=\pr_{E_E}t_y$.
Then $u_y=a_y+e_y$ in this identification, and
\begin{equation}\label{eq:first-moments}
\sum_y a_y=0,\qquad \sum_y e_y=0,\qquad
\norm{a_y}^2+\norm{e_y}^2=9/4.
\end{equation}
Integrality against roots places $a_y$ in $A_{11}^*$ and $e_y$ in $E_6^*$.

\subsection{Two short-vector types}
Realize $A_{11}$ as the integer zero-sum lattice on $12$ coordinates.
A vector of its dual has coordinates
\[
z_i-b/12,\qquad z_i\in\Z,\qquad \sum_i z_i=b,
\]
where one may choose $0\le b\le11$. Its norm is at least
$b(12-b)/12$, with equality exactly when all $z_i$ are $0$ or $1$;
its norm modulo $2\Z$ is $11b^2/12$.
The discriminant group of $E_6$ is cyclic of order $3$. The trivial
class has even norm; each nontrivial class has minimum $4/3$ and norms
congruent to $4/3$ modulo $2\Z$.
Combining these facts with~\eqref{eq:first-moments} leaves just
\begin{equation}\label{eq:raw-types}
\begin{array}{c|c|c}
b & \norm{a_y}^2 & \norm{e_y}^2\\ \hline
3\text{ or }9 & 9/4 & 0\\
1\text{ or }11 & 11/12 & 4/3.
\end{array}
\end{equation}
In each case the $A_{11}$ lower bound is attained, so the integer
coordinates are binary. These statements use short-vector inequalities,
not a list of vectors in a shell.

There is also a \emph{global} choice of orientation. If the $E_6$
class of index $y$ is denoted by $\epsilon_y\in\{0,1,-1\}$, integrality
of~\eqref{eq:lifted-pair} gives
\begin{equation}\label{eq:label-congruence}
12\mid 21-b_yb_z+16\epsilon_y\epsilon_z.
\end{equation}
Choose an index with nontrivial $E_6$ class, if one exists, and compare
all indices with it in this congruence. After possibly reversing all
$A_{11}$ coordinates, the two types become
\begin{equation}\label{eq:normalized-types}
\begin{array}{c|c|c|c}
\text{type} & b & \norm{a_y}^2 & e_y\\ \hline
\mathrm{I} & 9 & 9/4 & 0\\
\mathrm{II} & 1 & 11/12 &
 \norm{e_y}^2=4/3,\quad e_y\in\delta\omega_E+E_6,
\end{array}
\end{equation}
for one fixed $\delta\in\{1,-1\}$ and a generator $\omega_E$ of the
dual quotient. If all classes are trivial, the same congruence forces
all labels to be $3$, or all to be $9$, and the same normalization applies.
\leanref{SRG266/Lattice/ShortFrameLabelNormalization.lean}{shortFrameLabel_global_orientation}

\subsection{The type-\texorpdfstring{$A$}{A} alternative and root placement}
We shall prove in Section~\ref{sec:design} that the required
quasi-symmetric $2$-$(56,12,9)$ design does not exist. We use that result
here only through the following elementary implication.

\begin{lemma}\label{lem:type-a}
The family $(e_y)$ is not identically zero.
\end{lemma}
\begin{proof}
If every $e_y=0$, the normalized $A_{11}$ coordinates have binary
weight~$9$. Complement their supports to obtain binary columns $w_y$
of weight~$3$ on $12$ coordinates. Centering subtracts $(1/4)\one$,
and hence
\[
\ip{w_y}{w_z}=\ip{a_y}{a_z}+3/4=L_{yz}.
\]
Let $F$ be the resulting $12\times220$ matrix. Its rank is~$12$,
since $F^{\mathsf T}F=L$. The centroid is an integral vector $q$ in
these coordinates and $F^{\mathsf T}q=15\one$.
Because $F^{\mathsf T}(5\one)=15\one$ and $F^{\mathsf T}$ is injective,
we have $q=5\one$. Thus $F\one=11q=55\one$.
The support of any row is a $55$-coclique in $H$: an edge of $H$ has
zero Gram entry, whereas two columns sharing a row have positive product.
Adjoining $o$ gives a $56$-coclique in~$\Gamma$.

A coclique $C_0$ of this size meets the Hoffman bound. Its equality case
\cite[Section~3.5]{BrouwerHaemers2012} forces every outside vertex to
have $12$ neighbours in $C_0$. These neighbourhoods form a
$2$-$(56,12,9)$ design. This is an instance of the regular-set and
quasi-symmetric-design correspondence \cite{Neumaier1982,AdmEtAl2018}.
To obtain its intersections explicitly,
write $M_0$ for its incidence matrix and $A_0$ for the adjacency matrix
outside the coclique. The blocks of~\eqref{eq:srg} imply
\[
M_0M_0^{\mathsf T}=36I+9J,\qquad
M_0A_0=9J-9M_0,\qquad
M_0^{\mathsf T}M_0+A_0^2=36I-9A_0+9J.
\]
The symmetric matrix
$Z_0=M_0^{\mathsf T}M_0-(9I+3J-3A_0)$ has zero diagonal and satisfies
$Z_0^2=-45Z_0$. Its diagonal square entries are sums of squares, so $Z_0=0$.
Distinct design blocks therefore meet in $0$ or $3$ points, contradicting
Theorem~\ref{thm:noqsd}.
\end{proof}
\leanref{SRG266/Completion/TypeAFactorization.lean}{lorentzian_type_A_binary_factorization}

Each of the four complements in Table~\ref{tab:complements} contains
an $A_4$ chain. Connected root chains must lie wholly in one of the
two components $A_{11},E_6$: a root is supported in one component, and
adjacent roots have nonzero pairing. We next show that an $A_4$ chain
of the complement cannot lie in~$E_6$.

\begin{lemma}\label{lem:e6-gap}
Suppose an $A_4$ root chain lies in $E_6$. Vectors $v,w\in E_6^*$ of
norm $4/3$ in the same nontrivial discriminant class, both orthogonal
to this chain, satisfy $\ip{v}{w}\ge1/3$.
\end{lemma}
\begin{proof}
The pairing is congruent to $4/3$ modulo $\Z$. Cauchy--Schwarz bounds
it below by $-4/3$, so the only possibility below $1/3$ is $-2/3$.
In that case $v-w$ and $v+2w$ have Gram matrix
$\left(\begin{smallmatrix}4&-2\\-2&4\end{smallmatrix}\right)$.
Together with the orthogonal $A_4$ chain these six rational vectors
have Gram determinant $5\cdot12=60$.
Every rational basis of the $E_6$ space has determinant $3d^2$ for
some $d\in\Q$. This would give $d^2=20$, which is impossible.
\end{proof}

If a complementary $A_4$ lay in $E_6$, every $e_y$ would be orthogonal
to it. By the common-coset normalization and Lemma~\ref{lem:e6-gap},
all pairings among the nonzero $e_y$ would be positive. But their sum is
zero by~\eqref{eq:first-moments}. Pairing a nonzero member with that sum
is a contradiction. Thus they would all vanish, contrary to
Lemma~\ref{lem:type-a}. Every complementary $A_4$, and consequently
every complementary $A_5$, therefore lies in~$A_{11}$.

\subsection{Masses from a constant coordinate block}
An $A_4$ root chain in $A_{11}$ is, up to signed orientation and coordinate
permutation, a chain of differences on five of the twelve coordinates.
Orthogonality makes these five coordinates equal in every $a_y$.
In type~I the corresponding five binary entries must all be $1$,
because there are only three zero entries in total; in type~II they
must all be $0$, because there is only one entry equal to~$1$.
Their common centered coordinates are thus $1/4$ and $-1/12$.
If $n_0,n_1$ count the two types, the first moment yields
\[
n_0+n_1=220,\qquad n_0/4-n_1/12=0.
\]
It follows that
\begin{equation}\label{eq:masses}
n_0=55,\quad n_1=165,\qquad
\sum_y\norm{a_y}^2=275,\qquad
\sum_y\norm{e_y}^2=220.
\end{equation}
The argument applies to every complement in the table. It counts
\emph{indices}; no injectivity of the projection maps is used.
\leanref{SRG266/Lattice/A11FirstMomentCounts.lean}{a11_A4_first_moment_counts}

\subsection{Trace bounds eliminate three complements}
For $v\in E_A$ and $w\in E_E$, equation~\eqref{eq:tight-frame} and
orthogonal projection give
\begin{equation}\label{eq:component-bound}
\sum_y\ip{v}{a_y}^2\le45\norm{v}^2,\qquad
\sum_y\ip{w}{e_y}^2\le45\norm{w}^2.
\end{equation}
More generally, if a family in a $d$-dimensional Euclidean space satisfies
this bound and is orthogonal to a $t$-dimensional subspace, its total
squared norm is at most $45(d-t)$. This follows by summing the bound over
an orthonormal basis of the orthogonal complement.

Both $E_6$ and $H_6$ in Table~\ref{tab:complements} contain an $A_5$.
It must lie in $E_A$, making the $A$-mass at most $45(11-5)=270$,
contrary to~\eqref{eq:masses}.
For $C=A_2\perp A_4$, the $A_4$ again lies in~$E_A$.
If the $A_2$ lies in $E_E$, the $E$-mass is at most $45(6-2)=180$.
If it lies in $E_A$, the $A$-mass is at most $45(11-6)=225$.
Both contradict~\eqref{eq:masses}.
Hence
\begin{equation}\label{eq:last-complement}
C=A_4\perp Q_{15}.
\end{equation}

\subsection{The binary projection trace}
Let $r_1,r_2$ be the specified basis of $Q_{15}$, so
$\norm{r_1}^2=\norm{r_2}^2=4$ and $\ip{r_1}{r_2}=1$.
Write $v=\pr_{E_E}r_1$ and $w=\pr_{E_E}r_2$.
The $A_4$ part of $W$ is contained in $E_A$, so only these two
vectors contribute to the projection of $W$ onto~$E_E$.
The tight frame identity, now used as an exact identity rather than an
inequality, gives
\begin{align}
\sum_y\norm{e_y}^2
&=45\bigl(6-\tr(\pr_{E_E}\pr_W|_{E_E})\bigr)\nonumber\\
&=270-3\bigl(4\norm{v}^2+4\norm{w}^2-2\ip{v}{w}\bigr).
\label{eq:binary-trace}
\end{align}
The factor $1/15$ in the projection trace is the reciprocal of the
binary Gram determinant: the inverse matrix is
$\frac1{15}\left(\begin{smallmatrix}4&-1\\-1&4\end{smallmatrix}\right)$.
Equating~\eqref{eq:binary-trace} with the $E$-mass $220$ determines the
otherwise unknown projection angle.

\begin{lemma}\label{lem:binary-values}
There are integers $x,y,z$ with
\[
x=3\norm{v}^2,\quad y=3\norm{w}^2,\quad z=3\ip{v}{w},
\]
such that
\begin{equation}\label{eq:binary-constraints}
x,y\in\{0,4,6,12\},\qquad z^2\le xy,\qquad
x=y=6\ \Longrightarrow\ 3\mid z.
\end{equation}
\end{lemma}
\begin{proof}
The $E_6$ projections of $r_i\in\mathcal L$ belong to $E_6^*$.
Their norms are nonnegative and at most $4$, and are congruent to either
$0$ or $4/3$ modulo $2\Z$. The only preliminary possibilities are
$0,4/3,2,10/3,4$. A norm $10/3$ would leave an $A_{11}^*$ projection
of norm $2/3$, smaller than its least positive norm $11/12$.
Thus it is excluded. Three times any pairing in $E_6^*$ is integral,
and Cauchy--Schwarz gives the displayed inequality. Finally an $E_6^*$
vector of norm~$2$ lies in $E_6$ itself, so two such vectors have an
integral pairing; in the scaled notation $3\mid z$.
\end{proof}

\begin{proof}[Proof of Theorem~\ref{thm:main}]
Apply the preceding constructions to the hypothetical graph.
The type-$A$ alternative is excluded by Theorem~\ref{thm:noqsd}.
The trace bounds leave~\eqref{eq:last-complement}, and
\eqref{eq:binary-trace} becomes
\begin{equation}\label{eq:final-equation}
4x+4y-2z=50.
\end{equation}
Equivalently $z=2x+2y-25$.
Using $z^2\le xy$, the only pair from the four allowed values, up to
interchanging $x,y$, is $x=y=6$. For clarity the values of
$z^2-xy$ for the ten unordered pairs are
\[
\begin{array}{c|rrrrr}
(x,y)&(0,0)&(0,4)&(0,6)&(0,12)&(4,4)\\
z^2-xy&625&289&169&1&65\\ \hline
(x,y)&(4,6)&(4,12)&(6,6)&(6,12)&(12,12)\\
z^2-xy&1&1&-35&49&385
\end{array}
\]
For $x=y=6$, equation~\eqref{eq:final-equation} gives $z=-1$, contradicting
$3\mid z$. This completes the contradiction.
\end{proof}
\leanref{SRG266/Completion/MomentBounds.lean}{no_q15_projection_trace}

The last arithmetic check has only four possible values per norm. It is
not a substitute encoding of the removed shell searches: its inputs
are the projection norms of two specified complementary vectors, obtained
from the actual marked completion and its frame trace.

\section{The design obstruction without classification}\label{sec:design}

This section is independent of the rank-$24$ construction. Its input is
only a hypothetical quasi-symmetric design with the indicated parameters.

\begin{theorem}[Munemasa--Tonchev, with an internal alternative proof]
\label{thm:noqsd}
There is no quasi-symmetric $2$-$(56,12,9)$ design with block
intersection numbers $0$ and $3$.
\end{theorem}

The statement is due to Munemasa and Tonchev \cite{MunemasaTonchev2020}.
We give the argument used to prove it inside the formalization.
The central point is that local triangular reconstructions can be made
compatible across all $56$ choices of a deleted point. The resulting
canonical biplane extends the hypothetical design by one point. A local
graph obstruction then replaces the biplane and code classifications.

\subsection{Local triangular graphs}
Let $\mathcal D$ be the hypothetical design on a point set $P_0$.
It has $210$ blocks and replication number~$45$.
Fix $p\in P_0$. The $45$ blocks through $p$, with $p$ deleted,
have size~$11$, pairwise intersection~$2$, and replication~$9$
on the remaining $55$ points. Their dual is a $2$-$(45,9,2)$ design.
For $q,r\ne p$ let $m_p(q,r)$ count blocks through $p,q,r$;
set $m_p(q,q)=9$.
Double counting gives
\[
\sum_{r\ne p,q}m_p(q,r)=90,\qquad
\sum_{r\ne p,q}m_p(q,r)^2=162.
\]
Since $(t-1)(t-2)\ge0$ for every nonnegative integer $t$, the sum
of these expressions over $54$ points is
$162-3\cdot90+2\cdot54=0$. Every off-diagonal multiplicity is
therefore $1$ or~$2$.

Join $q,r$ when $m_p(q,r)=1$. There are $18$ neighbours of each
vertex, and the mixed moment
\[
\sum_{s\ne p}m_p(q,s)m_p(r,s)=162+9m_p(q,r)
\]
gives $14-5m_p(q,r)$ common neighbours. The link graph thus has
parameters $(55,18,9,4)$.
It is the triangular graph $T(11)$. The needed uniqueness result
\cite{Connor1958} is proved internally: its neighbourhoods split into
two $9$-cliques with a matching between them; the resulting grand
$10$-cliques number $11$, each vertex belongs to two, and pairs of
grand cliques meet in one vertex. Assigning to a vertex its two grand
cliques gives the required edge coordinates.
Choose an identification
\[
\eta_p:P_0\setminus\{p\}\longrightarrow\binom{[11]}2.
\]
For distinct $p,q,r$ it satisfies
\begin{equation}\label{eq:triple-link}
\kappa(p,q,r):=\#\{B\in\mathcal D:p,q,r\in B\}
=2-|\eta_p(q)\cap\eta_p(r)|.
\end{equation}
The left side is symmetric in the three points, providing compatibility
between the different choices of local coordinates.

\subsection{Ternary self-orthogonality glues the stars}
For $a\in[11]$ define the closed star
\[
S_{p,a}=\{p\}\cup\{q\ne p:a\in\eta_p(q)\}.
\]
It has $11$ points. Let $\mathcal C\subseteq\F_3^{P_0}$ be the code
spanned by the incidence vectors of the original blocks.
Because both block size and all block intersections are divisible by~$3$,
\begin{equation}\label{eq:self-orthogonal}
\mathcal C\subseteq\mathcal C^\perp.
\end{equation}
The following observation uses this code without enumerating its words.

\begin{lemma}\label{lem:stars}
Every complemented star $\one-\one_{S_{p,a}}$ lies in $\mathcal C$.
The closed stars, after duplicates are identified, are the blocks of a
canonical $2$-$(56,11,2)$ biplane $\mathcal B$.
They satisfy
\begin{align}
|B\cap S|&\in\{0,3\}\quad(B\in\mathcal D,\ S\in\mathcal B),
\label{eq:biplane-cross}\\
\kappa(p,q,r)+\#\{S\in\mathcal B:p,q,r\in S\}&=2
\quad(p,q,r\text{ distinct}).\label{eq:biplane-triple}
\end{align}
\end{lemma}

\begin{proof}
For $q\ne p$ take the sum in $\F_3^{P_0}$ of the original block
vectors containing $p,q$. Its coordinate at $x$ is $\kappa(p,q,x)$,
with value $9=0$ at $x=p,q$. If $\eta_p(q)=\{a,b\}$,
equation~\eqref{eq:triple-link} writes this word as
$2j_p-s_{p,a}-s_{p,b}$, where $j_p$ is $1$ outside $p$ and
$s_{p,a}$ is the open-star vector. Combining the three words for
$\{a,b\},\{a,c\},\{b,c\}$ gives $j_p-s_{p,a}$ in the code.
This is the complemented closed-star vector.
Self-orthogonality now implies
\begin{equation}\label{eq:star-mod3}
|S_{p,a}\cap S_{q,b}|\equiv2\pmod3.
\end{equation}

We detail why a closed star can be re-rooted at any one of its points;
this prevents treating unrelated local biplanes as a single object.
Let $S=S_{p,a}$ and $q\in S\setminus\{p\}$.
In the $q$-coordinates, the ten edges representing $S\setminus\{q\}$,
including the central edge itself, all meet the central edge
representing~$p$. By~\eqref{eq:star-mod3},
every coordinate degree of this ten-edge set is $1$ modulo~$3$.
A vertex outside the central edge has degree at most~$2$, hence degree
exactly~$1$. Thus two noncentral edges meet precisely when they attach
to the same end of the central edge. For triples on
$S\setminus\{p\}$ this gives the two-colour link parity identity.
On four distinct points, write $A,B,C,D$ for the assertions that each
of its four triples has multiplicity~$1$. The three link identities
are $A\leftrightarrow(B\leftrightarrow C)$,
$A\leftrightarrow(B\leftrightarrow D)$, and
$A\leftrightarrow(C\leftrightarrow D)$; together they force $A$.
Using a fourth point of $S\setminus\{p\}$ proves that every triple
in $S$ has multiplicity~$1$.
Consequently, in coordinates at any $q\in S$, all ten edges are
pairwise intersecting. More than three pairwise intersecting edges of
a complete graph form a star. Hence $S=S_{q,b}$ for some~$b$.

Two closed stars have nonempty intersection by~\eqref{eq:star-mod3}.
Re-rooting at a common point shows that distinct ones meet in exactly
two points. At each point there are exactly $11$ stars, and each pair
of points lies in two. Counting point--star incidences gives $56$
distinct stars. These are the stated biplane.

For a block $B$ containing $p$, its $11$ edges in the $p$-coordinates
have degree~$2$ at each of the $11$ coordinates. This follows from the
local incidence moments (and is proved before any star is glued).
If $B\cap S$ is nonempty, re-root at $p\in B\cap S$; the intersection
then consists of $p$ and the two edges at the star centre, giving~$3$.
Finally the number of biplane blocks containing $p,q,r$ is
$|\eta_p(q)\cap\eta_p(r)|$, so~\eqref{eq:biplane-triple} follows
from~\eqref{eq:triple-link}.
\end{proof}
\leanref{SRG266/Design/CanonicalBiplane.lean}{canonicalBiplane_triple_complement}

Adjoin a new point $\infty$ and extend the $56$ biplane blocks by it,
retaining the $210$ original blocks. Equations
\eqref{eq:biplane-cross}--\eqref{eq:biplane-triple} give a
$3$-$(57,12,2)$ design $\mathcal T$ with $266$ blocks and block
intersections $0,3$. Triples involving $\infty$ use the biplane pair
equation; the others use~\eqref{eq:biplane-triple}.
Its replication number is $56$ and its pair multiplicity is~$11$.
This is a construction from $\mathcal D$, not an appeal to a classification
of all biplanes or of their ternary codes.

\subsection{The associated Krein graph}
We use the classical point--block construction for triangle-free strongly
regular graphs; see \cite[Chapter~5]{CameronVanLint1980} and the proof of
the Krein-graph characterization in \cite{ChenMakhnevNirova2023}.
Construct a graph $K$ on the disjoint union of one new vertex $*$,
the $57$ points of $\mathcal T$, and its $266$ blocks. Join $*$ to
all points, join a point to its incident blocks, and join two blocks
when they are disjoint. There are no other edges.
If $M_1$ is the incidence matrix and $A_{\mathrm{disj}}$ is the block
disjointness matrix, then
\begin{align*}
M_1M_1^{\mathsf T}&=45I+11J,\qquad
M_1^{\mathsf T}M_1=9I+3J-3A_{\mathrm{disj}},\\
M_1A_{\mathrm{disj}}&=12(J-M_1).
\end{align*}
These equations give, for the adjacency matrix $A_K$,
\begin{equation}\label{eq:krein-srg}
A_K^2=45I-12A_K+12J.
\end{equation}
Thus $K$ is an $\SRG(324,57,0,12)$.
The triple-design equation further gives
\begin{equation}\label{eq:independent-triple}
|\Gamma_K(a)\cap\Gamma_K(b)\cap\Gamma_K(c)|=3
\quad\text{for every independent triple }\{a,b,c\}.
\end{equation}
For example, if all three vertices are points the common neighbours
are $*$ and their two incident triple blocks. The other vertex types
follow by expanding products of incidence entries and using the same
triple equation; these expansions are formalized in
\path{SRG266/Design/Krein/TripleMoments.lean}.

We now use only the local identities available for this constructed graph.
For an independent triple $\tau$, put
$\tau^\circ=\bigcap_{v\in\tau}\Gamma_K(v)$.
Then $\tau^\circ$ is an independent
triple and $(\tau^\circ)^\circ=\tau$. The six vertices
$\Delta=\tau\cup\tau^\circ$ induce $K_{3,3}$.
Every vertex outside $\Delta$ has $0$, $1$, or $2$ neighbours in it:
it cannot meet both sides without a triangle, and meeting all three
vertices of one side would put it in the other side.
Partition these outside vertices into $X_0,X_1,X_2$ by this number.

\begin{lemma}[The cubic pair]\label{lem:cubic-pair}
One has $|X_0|=48$, $|X_1|=216$, and $|X_2|=54$.
The induced graphs on $U_0=\Delta\cup X_0$ and $U_2=X_2$ are
cubic graphs on $54$ vertices. Each vertex has $18$ neighbours in the
other set. If $A_0,A_2$ are their adjacency matrices and $B$ is the
cross-incidence matrix, then
\begin{equation}\label{eq:cubic-gram}
BB^{\mathsf T}=(A_0-3I)^2+6J,\qquad
B^{\mathsf T}B=(A_2-3I)^2+6J.
\end{equation}
\end{lemma}
\begin{proof}
Counting vertices, edges to $\Delta$, and common-neighbour pairs in
$\Delta$ gives
\[
|X_0|+|X_1|+|X_2|=318,\quad |X_1|+2|X_2|=324,
\quad |X_2|=54.
\]
For the last count there are six nonadjacent pairs in $\Delta$,
each with $12$ common neighbours, of which three are inside $\Delta$.
The stated sizes follow.
The analogous counts with one vertex fixed, using
\eqref{eq:independent-triple}, give the local degrees $3$ and~$18$.
For distinct nonadjacent $x,y\in U_i$, they also give
\[
|\Gamma(x)\cap\Gamma(y)\cap U_{2-i}|
=6+|\Gamma(x)\cap\Gamma(y)\cap U_i|.
\]
For adjacent vertices both common-neighbour counts vanish; on the
diagonal they are $18$ and $3$. These three cases are exactly
\eqref{eq:cubic-gram}.
\end{proof}

Taking traces of the squares in~\eqref{eq:cubic-gram} gives
$\tr A_0^4=\tr A_2^4$. Indeed the two Gram products have equal square
traces, while both cubic graphs have $54$ vertices, trace zero,
second trace $162$, and third trace zero.
A cubic graph on $n$ vertices without four-cycles has fourth trace
$15n$. The graph on $U_0$ contains $\Delta=K_{3,3}$, so its fourth
trace is strictly greater than $15\cdot54$. The graph on $X_2$ must
therefore contain a four-cycle.

\begin{lemma}[Unique companion]\label{lem:companion}
There is exactly one induced $K_{3,3}$ whose vertices lie in $X_2$.
Together with $\Delta$ it induces $K_{6,6}$ minus a perfect matching.
\end{lemma}
\begin{proof}
We first record the two local counting obstructions that make the
four-cycle sufficient.
For opposite vertices $x,y$ of a four-cycle in $X_2$, let
$\nu_{xy}=|\Gamma(x)\cap\Gamma(y)\cap X_2|$ and
$d_{xy}=|\Gamma(x)\cap\Gamma(y)\cap\Delta|$.
Then $2\le\nu_{xy}\le3$ and $d_{xy}\ge1$. For
$Z_{xy}=(\Gamma(x)\cup\Gamma(y))\cap X_0$, the local counts give
\begin{equation}\label{eq:square-expansion}
|Z_{xy}|=26+d_{xy}-\nu_{xy}.
\end{equation}
If $f,g$ are the other opposite vertices, triangle-freeness makes
$Z_{xy}$ and $Z_{fg}$ disjoint. If $\nu_{xy}=2$, their sizes would be
at least $25$ and $24$, contradicting $|X_0|=48$.
Thus every pair in $X_2$ with at least two common neighbours there
has three. Cubicity now completes the four-cycle to a $K_{3,3}$
component of the graph on $X_2$.

Second, there is no $K_{3,3}$ inside $X_0$.
For nonadjacent distinct $x,y\in X_0$, the local common counts give
\[
|\Gamma(x)\cap\Gamma(y)\cap X_1|
+2|\Gamma(x)\cap\Gamma(y)\cap X_0|=6.
\]
If the second common-neighbour count were three, the first would
vanish. All common neighbours would then have an even number of
neighbours in $\tau$ (zero or two). But summing their degrees into $\tau$
counts, for each $s\in\tau$, the three common neighbours of the
independent triple $x,y,s$, and so gives the odd number~$9$.
This contradiction bounds common-neighbour counts inside $X_0$ by~$2$
and excludes a $K_{3,3}$ there.

Two distinct $K_{3,3}$ components inside $X_2$ would be disjoint and
have no edges between them. Relative to either biclique, the other
would lie in its zero-neighbour class, contradicting the preceding
obstruction. This proves uniqueness.
Every vertex of the companion has two neighbours in $\Delta$.
Counting the $12$ cross edges and using the outside-degree bound shows
that every vertex of $\Delta$ also has two neighbours in the companion.
Triangle-freeness makes the two bipartitions compatible. After choosing
their orientations, the union has two independent sides of size~$6$,
and every vertex meets five vertices on the other side. The missing
edges are consequently a perfect matching.
\end{proof}
\leanref{SRG266/Design/Krein/Companion.lean}{exists_unique_biclique_companion}

Call an induced $K_{6,6}$ minus a perfect matching a \emph{crown}.
Every independent triple lies in the crown constructed above, and this
crown is unique. To see uniqueness, an independent triple in any crown
lies in one side: a vertex on one side has only one non-neighbour on
the other, so an independent triple cannot meet both sides.
Its three common neighbours are the opposite side minus the three
matched vertices. The remaining six vertices form a biclique in its
class $X_2$, hence must be its unique companion.

\subsection{Steiner completion and divisibility}
Choose distinct nonadjacent vertices $a,b$ of $K$ and let
$\Omega=\Gamma_K(a)\cap\Gamma_K(b)$. This is an independent set of size~$12$.
For each crown $\mathcal W$ meeting $\Omega$ in four vertices, take
$\mathcal W\cap\Omega$ as a
block, identifying repeated subsets.
Every triple $\tau\subset\Omega$ lies in its unique crown~$\mathcal W$.
In a bipartition of $\mathcal W$, say $\tau$ lies on the left, its common neighbours
are exactly the three unmatched vertices on the right. Both $a,b$
belong to this triple of common neighbours.
Their common neighbours in the crown are the four vertices on the
left other than their two distinct matching partners. Hence
$|\mathcal W\cap\Omega|=4$.
Any other block containing $\tau$ would come from another crown containing
the same independent triple, which is impossible. We have constructed
a Steiner $3$-$(12,4,1)$ design.

Fix a point in such a design. The $\binom{11}{2}=55$ triples containing
that point must be partitioned into the $\binom32=3$ triples contributed
by each block through it. Its integral replication number $r$ would satisfy
$3r=55$. This contradiction proves Theorem~\ref{thm:noqsd}.
\leanref{SRG266/Design/Krein/SteinerCompletion.lean}{no_quasi_symmetric_design}

The intervening $3$-$(57,12,2)$ design was already excluded by
Kaski and {\"O}sterg{\aa}rd \cite[Corollary~2]{KaskiOstergard2008},
combining their classification with the nonextendability computations
of Key and Tonchev \cite{KeyTonchev1997}. Bagchi's earlier claimed
proof \cite{Bagchi1988}, with its subsequent corrigendum, contained a
linear-algebra error; Key and Tonchev discuss the error and credit
Brouwer with finding it. We do not use that argument.
The general nonexistence of the Krein graph with parameters
$(324,57,0,12)$ was proved by Gavrilyuk and Makhnev
\cite{GavrilyukMakhnev2005}.
The purpose of these constructions is to close the required dependency
inside Lean with local moment, parity, and divisibility arguments.

\section{Formal verification and reproducibility}\label{sec:verification}

\subsection{Statement and trusted boundary}
The development uses Lean~4 \cite{deMouraUllrich2021} and Mathlib
\cite{Mathlib2020}. The public theorem in \path{Main.lean} is
\begin{quote}\small\ttfamily
theorem SRG266.nonexistence\\
\hspace*{1em}\{V : Type u\} [Fintype V] (G : SimpleGraph V)\\
\hspace*{1em}[DecidableRel G.Adj] :\\
\hspace*{1em}$\neg$ G.IsSRGWith 266 45 0 9
\end{quote}
Here \code{IsSRGWith} is Mathlib's own definition, not a weakened
project-specific substitute. The type $V$ is arbitrary and universe
polymorphic. The finite-type and decidable-adjacency instances support
finite counting; no graph-theoretic hypothesis is hidden in a typeclass.

The axiom collector for the final theorem reports exactly
\[
\text{\code{propext}},\qquad \text{\code{Classical.choice}},\qquad
\text{\code{Quot.sound}}.
\]
There are no added mathematical axioms, assumed host lists, unproved
design obstructions, or external solver answers in this dependency
closure. Tactics construct proof terms checked by the kernel. The
source audit rejects proof holes and the prohibited external-evaluation
mechanisms, and checks that all $258$ project proof modules are reachable
from \path{Main.lean}. The axiom check also elaborates the public statement
again as an example, so it checks the theorem's type as well as its axioms.

We use \emph{certificate-free} in the following specific sense: the build
does not consume external Hall, Farkas, SAT/LRAT, shell-enumeration, or
similar infeasibility certificates. Small finite arithmetic arguments
remain, for example the four norm values in Lemma~\ref{lem:binary-values}
and the primary discriminant normal forms. Naturally the Lean proof term
itself is a certificate in the general proof-theoretic sense. We do not
claim that finite case distinctions have disappeared.

Kernel checking establishes the formal theorem under the displayed
axioms. It does not independently certify the accuracy of every sentence
in this paper, the historical attribution, or the absence of hardware
and implementation faults. These are separate responsibilities.

\begin{samepage}
\subsection{Archived source and build procedure}
Release \textbf{v2.0.0} of the formalization is archived at
DOI \href{https://doi.org/10.5281/zenodo.22509839}{10.5281/zenodo.22509839}
\cite{Akiyama2026}. Its Git commit is
\begin{center}\small
\path{5fb00e6db138d6df962d71b3b2e1a26ac42a9e59}.
\end{center}
\end{samepage}
The archive contains the proof source; it is not an archive of this
manuscript. The Lean toolchain is pinned to \code{leanprover/lean4:v4.33.1}.
Mathlib is pinned to the following commit, corresponding to its tag
\code{v4.33.1}:
\begin{center}\small
\path{0df444a360eaa60ab8c11dca51a86af692955474}.
\end{center}
After installing elan and obtaining that release, a bounded local build
can be run as follows:
\begin{verbatim}
export LEAN_NUM_THREADS=8
lake exe cache get
python3 scripts/audit.py
lake build
lake env lean scripts/check_axioms.lean
\end{verbatim}
During \code{lake build}, the repository passes \code{-j1} and
\code{-M4096} to project elaboration subprocesses.
Limiting a subprocess to one thread is not by itself a
limit on the number of simultaneous Lake jobs; the
\code{LEAN_NUM_THREADS} environment variable must also be set.
The continuous-integration workflow has one verification job, with a
thread budget of~$2$, and performs the cache retrieval, source audit,
build, and axiom check. It does not distribute project object files
between jobs.

For a Linux system supporting user-systemd memory control, the measured
eight-core, $32$-GiB-budget build used the following form of command,
after the Mathlib cache had been obtained:
\begin{verbatim}
lake clean srg266
systemd-run --user --quiet --wait --pipe \
  --property=MemoryMax=32G --property=MemorySwapMax=0 \
  --working-directory="$PWD" \
  /usr/bin/taskset -c 0-7 /usr/bin/env LEAN_NUM_THREADS=8 \
  python3 scripts/measure.py .lake/measured-build lake build
\end{verbatim}
The CPU numbers should be adjusted to the machine's topology and allowed
CPU set. In the reported run, CPUs $0$--$7$ represented eight distinct
physical cores. Affinity enforced the CPU bound and cgroup memory
control enforced the memory bound; a declared quota without an active
controller is not evidence of enforcement.

\begin{table}[ht]
\centering
\caption{Local verification measurements, 6 September 2026. Project
object files were cleaned for the Lean build; the Mathlib cache was
already present. The nanoda measurement excludes the preceding export.}
\label{tab:measurements}
\begin{tabular}{@{}lll@{}}
\toprule
 & Lean project build & nanoda check\\
\midrule
Wall time & $884.17$ s ($14$ min $44$ s) & $24.47$ s\\
Largest child-process RSS & $3.85$ GiB & $1.81$ GiB\\
Cgroup peak memory & $3.68$ GiB & $1.81$ GiB\\
CPU affinity & CPUs $0$--$7$ & CPUs $0$--$7$\\
Memory / swap limits & $32$ GiB / $0$ & $32$ GiB / $0$\\
OOM events & $0$ & $0$\\
\bottomrule
\end{tabular}
\end{table}

The machine was an Intel Core i9-10900KF desktop with $64$~GB installed
RAM, running Linux~6.8.0-138. The limit was a build budget, not the
observed requirement. Maximum child RSS is a per-process high-water
mark, not aggregate memory; cgroup accounting is a different measurement
and the two values need not coincide. Cache download time and a rebuild
of Mathlib from source are not included. Source hashes were recorded
and checked unchanged by the measurement script; the proof sources of
the archived release agree with this measured source set.

\subsection{Independent checking}
We exported the dependency closure of \code{SRG266.nonexistence} with
\code{lean4export} and checked it with the independent checker nanoda
\cite{Lean4Export,Nanoda}. The check accepted $77{,}402$ declarations
with only the same three permitted axioms and with unsafe axiom
allowances disabled. The configuration \path{scripts/nanoda.json} enables
\code{nat_extension} and \code{string_extension}, so the check uses
nanoda's implementations of the Nat and String literal extensions
supported by the Lean kernel. These extensions belong to the checker's
trusted implementation, rather than adding declared axioms.
The pinned revisions used were
\begin{center}\small
\begin{tabular}{@{}ll@{}}
\code{lean4export} & \path{411dce7db58a3afc60ecab2d211acd1042b593dc}\\
nanoda & \path{05055695879dfebb6628a67da88ceca6cd6b0421}.
\end{tabular}
\end{center}
The exporter's toolchain pin was adjusted to Lean~4.33.1 before building
it. Its export format and the selected nanoda revision were checked for
compatibility. With those executables built, the check takes the form
\begin{verbatim}
LEAN_NUM_THREADS=1 lake env /path/to/lean4export \
  Main -- SRG266.nonexistence > .lake/nonexistence.ndjson
/path/to/nanoda_bin scripts/nanoda.json \
  < .lake/nonexistence.ndjson
\end{verbatim}
The exported stream had SHA-256
\begin{center}\footnotesize
\path{17b9cf9e04de998981a10d1d0a11f854b7b0aebbf3b9a5d59bcf07933d0a94ce}.
\end{center}
This is a reproducibility fingerprint of the measured export, not an
additional mathematical assumption. Independent checking reduces reliance
on a single kernel implementation; it does not certify the exporter,
checker, and hardware as infallible.

\section{Use of large language models}\label{sec:disclosure}

This project is human-directed and LLM-led in mathematical research and
formalization. Kay Akiyama selected the problem and supplied the objective
and continuing instructions, and is the author and maintainer responsible
for the work. The models were used for strategy, mathematical exploration,
formal proof construction, debugging, verification support, and exposition,
not merely language polishing.

The release disclosure attributes the initial mathematical research to
GPT-5.6 Sol and Claude Fable~5, and the initial Lean formalization to
Claude Opus~5 with Claude Code and GPT-5.6 Sol with Codex.
GPT-6 Astra carried out the proof-chain contraction and the formalization work
that eliminated the large external certificates. This included the
certificate-free lattice and design arguments described above.
GPT-5.6 Sol assisted with manuscript preparation and review.
GPT-6 Astra prepared the present exposition using the formal source and
primary literature.

These are workflow-level attributions, not a verified line-by-line
allocation of individual lemmas to particular models. No LLM is listed
as an author. The correctness claim for the formal theorem rests on
proof checking, not on a model's account of its reasoning. The author
remains responsible for the exposition, citations, disclosure, and
submission.

\section{Conclusion}

The obstruction is obtained by preserving a local frame while changing
its ambient geometry. The integral centroid provides the arithmetic
needed for a marked even unimodular completion. A root-isolation gap
and a vanishing harmonic theta series determine its root system.
The first moment fixes the two shell-type counts, while the second
moment bounds or determines the remaining projections. The terminal
contradiction concerns two vectors, rather than a large enumerated shell.
The design branch follows the same principle: compatibility across
local structures replaces separate searches, and a forced completion
ends in an elementary divisibility obstruction.

The result is an end-to-end proof of the nonexistence theorem with no
external infeasibility data. The analytic and structural steps enlarge
the mathematics developed inside Lean, but remove the dependence on
large problem-specific certificates.

\appendix
\section{A map from the argument to the formal source}\label{app:formal-map}

All paths below are relative to the archived repository. This is a map
of the principal interfaces, not a claim that each step occupies only
the displayed file. The general-purpose lattice infrastructure is in
\path{SRG266/Lattice/}; the application-specific assembly is in
\path{SRG266/Completion/}.

{\small
\begin{longtable}{@{}>{\raggedright\arraybackslash}p{0.28\textwidth}
  >{\raggedright\arraybackslash}p{0.67\textwidth}@{}}
\toprule
Mathematical step & Formal entry point\\
\midrule
\endfirsthead
\toprule
Mathematical step & Formal entry point\\
\midrule
\endhead
Local matrix algebra &
\path{SRG266/LocalAlgebra.lean}, \path{SRG266/LocalSpectrum.lean}\\[3pt]
Integral centroid &
\path{SRG266/KernelReduction.lean}, \code{exists_integral_centroid}\\[3pt]
Lorentzian change and frame &
\path{SRG266/Completion/Lift.lean}, \path{SRG266/Completion/Perpendicular.lean}\\[3pt]
Rank-six complement and marked completion &
\path{SRG266/Lattice/RankSixComplements.lean},
\path{SRG266/Completion/Rank24Completion.lean}\\[3pt]
$D_7$ isolation &
\path{SRG266/Lattice/D7FrameIsolation.lean}, \code{d7_root_separation_of_frame}\\[3pt]
Harmonic theta and roots &
\path{SRG266/Lattice/Rank24HarmonicThetaAnalytic.lean},
\path{SRG266/Lattice/LevelOneWeight14.lean},
\path{SRG266/Completion/RootMoment.lean}\\[3pt]
Internal ADE analysis &
\path{SRG266/Lattice/ADEConnectedClassification.lean},
\path{SRG266/Completion/RootLayout.lean}\\[3pt]
Orientation and first moments &
\path{SRG266/Lattice/ShortFrameLabelNormalization.lean},
\path{SRG266/Lattice/A11FirstMomentCounts.lean}\\[3pt]
Complement placement and masses &
\path{SRG266/Completion/ResidualRootChain.lean},
\path{SRG266/Completion/FrameMassConstruction.lean},
\path{SRG266/Completion/Order75FrameMass.lean}\\[3pt]
Trace exclusions &
\path{SRG266/Completion/A5TraceObstruction.lean},
\path{SRG266/Completion/A2A4TraceObstruction.lean}\\[3pt]
Binary endpoint &
\path{SRG266/Completion/Q15ProjectionData.lean},
\path{SRG266/Completion/Q15TraceObstruction.lean}\\[3pt]
Binary factorization and coclique &
\path{SRG266/Completion/TypeAFactorization.lean},
\path{SRG266/NotOneIntegrable.lean}, \path{SRG266/CocliqueDesign.lean}\\[3pt]
Triangular reconstruction &
\path{SRG266/Design/TriangularUniqueness.lean}\\[3pt]
Canonical biplane and extension &
\path{SRG266/Design/CanonicalBiplaneCode.lean},
\path{SRG266/Design/ClosedStarGluing.lean},
\path{SRG266/Design/OnePointExtension57.lean}\\[3pt]
Krein graph and triple moments &
\path{SRG266/Design/Krein/Construction.lean},
\path{SRG266/Design/Krein/TripleMoments.lean}\\[3pt]
Cubic traces and companion &
\path{SRG266/Design/Krein/CubicGram.lean},
\path{SRG266/Design/Krein/FourthTrace.lean},
\path{SRG266/Design/Krein/Companion.lean}\\[3pt]
Steiner obstruction &
\path{SRG266/Design/Krein/SteinerCompletion.lean},
\code{no_quasi_symmetric_design}\\[3pt]
Final statement &
\path{Main.lean}, \code{SRG266.nonexistence}\\
\bottomrule
\end{longtable}
}

\bibliographystyle{amsplain}
\bibliography{references}
\end{document}